\documentclass[12pt,a4paper,twoside]{article}

\usepackage[applemac]{inputenc}
\usepackage[T1]{fontenc}
\usepackage[english]{babel}

\usepackage{fancyhdr}
\usepackage{indentfirst} 
\usepackage{newlfont}
\usepackage{verbatim}
\usepackage{enumerate}
\usepackage{hyperref}

\usepackage{amsthm}  
\usepackage{amssymb}  
\usepackage{amscd}
\usepackage{amsmath} 
\usepackage{latexsym} 
\usepackage{amsfonts}
\usepackage{amsbsy}   
\usepackage{mathrsfs} 
\usepackage{arcs}
\usepackage{yhmath}

\usepackage{graphicx}
\usepackage{color}
\usepackage{subfigure}

\theoremstyle{plain}
\newtheorem{teo}{Theorem}
\newtheorem{prop}[teo]{Proposition}

\theoremstyle{definition}

\newtheorem{ese}[teo]{Example}

\theoremstyle{remark}
\newtheorem{oss}[teo]{Remark}

\newcommand{\rp}[1]{\ensuremath{\mathbb{RP}^{#1}}}
\newcommand{\s}[1]{\ensuremath{\mathbf{S}^{#1}}}
\newcommand{\R}[1]{\ensuremath{\mathbb{R}^{#1}}}

\begin{document}

\title{On the KBSM of links in lens spaces}

%\footnote{Work supported by ...}
%\title{The lift of knots and links in lens spaces is not a complete invariant}

\author{Bo\v{s}tjan Gabrov\v{s}ek, Enrico Manfredi}

\maketitle

\begin{abstract}

In this paper the properties of the Kauffman bracket skein module of $L(p,q)$ are investigated. 
Links in lens spaces are represented both through band and disk diagrams. The possibility to transform between the diagrams enables us to compute the Kauffman bracket skein module on an interesting class of examples consisting of inequivalent links with equivalent lifts in the $3$-sphere. The computation show that the Kauffman bracket skein module is an essential invariant, that is, it may take different values on links with equivalent lifts. We also show how the invariant is related to the Kauffman bracket of the lift in the $3$-sphere.% modulo an ideal in the ring $\Z [ A^{\pm1}]$.

\noindent
 {{\it Mathematics Subject
Classification 2010:} Primary 57M25, 57M27; Secondary 57M10.\\
{\it Keywords:} knots/links, lens spaces, lift, skein module, Kauffman bracket.}\\

\end{abstract}

\begin{section}{Introduction}

Skein modules play an important role in geometric topology since they capture essential information about the geometry of 3-manifolds. For instance, they reflect the interaction between embedded 1-dimensional and embedded 2-dimensional submanifolds, in particular, the existence of an embedded non-separating surface produces torsion in the module \cite{P}.

Our focus will be the Kauffman bracket skein module (KBSM) of links in lens spaces $L(p,q)$. This module is of particul interest since it is a non-trivially finitely generated skein module. The motivation for studying links in lens spaces can also be justified in recent applications to biology \cite{BM} and physics \cite{Ste}.

As for the paper itself, it represents a merge of the theses of the first author \cite{Gb} and second author \cite{Mn}.

When studying knots in lens spaces, we can immediately notice that several different knot representations are scattered throughout literature. For example, band diagrams were originally used to calculate the KBSM of $L(p,q)$ \cite{HP}. Closely related punctured disk diagrams were used to tabulate knots in $L(p,q)$ \cite{Gb}. It was recently shown in \cite{CMM} that it is possible to generalize Drobotukhina's disk diagram for the projective space to the case of $L(p,q)$. 
%These diagrams are in a way perhaps the most natural ones to consider, since they avoid the cumbersome band move arising from the rational surgery description of $L(p,q)$ and follow the algebraic description of $L(p,q)$ more closely.
Disk diagrams, which follow the algebraic description of $L(p,q)$ more closely, allow us to construct a straightforward Wirtinger-type presentation of the link group \cite{CMM} and allow us to study lifts of knots in $L(p,q)$ to the cyclic cover $S^3$ \cite{Mn2}.

In order to better understand different representations of knots in $L(p,q)$, we explicitly show in Section 2 how to transform one diagram to the other.

In section 3 we show how to construct the lift of a link in $L(p,q)$ to a link in $S^3$ using band diagrams. In addition, we show that the KBSM is an essential invariant, that is to say, the KBSM may have different values on links with equivalent lifts. This, on the one hand, provides another evidence that the links considered in \cite{Mn2}, with equivalent lift, are actually distinct, and on the other hand, it shows that the KBSM is a genuine $L(p,q)$ link invariant that cannot be simply derived from an $S^3$ link invariant of the lift.

At last, in Section 4, following an idea of Chbili \cite{C2} about the Jones polynomial of freely periodic knots, we show the connection between the KBSM of a link in a lens space and the KBSM of its lift in $\s3$

\end{section}

%============================================================

\section{Different representations of links in lens spaces}\label{bandsection}

The aim of this section is to recall the possible representations of links in lens spaces, namely, the disk diagram of \cite{CMM}, the band diagram of \cite{HP}, and the closely related punctured disk diagram of \cite{Gb}. We show how to switch from one representation to the other.

\paragraph{Basic definitions}
Let $L(p,q)$ denote the lens space obtained by a $p/q$ rational surgery on the unknot in the $3$-sphere $\s3$. Here $p$ and $q$ are coprime integers satisfying $0 \leq q < p$.
% numbers (the notation is fixed as in \cite{PS}).

A $n$-component link $L$ in a $L(p,q)$ is the embedding of $n$ copies of $\s1$ to $L(p,q)$.
%A link $L$ in a $L(p,q)$ is a pair $(L(p,q),L)$, where $L$ is a submanifold of $L(p,q)$ diffeomorphic to the disjoint union of $\nu$ copies of $\s1$, with $\nu > 0$. We call \emph{component} of $L$ each connected component of the topological space $L$. When $\nu=1$ the link is called a \emph{knot}.
%We usually refer to $L \subset L(p,q)$ meaning the pair $(L(p,q),L)$.
A link $L \subset L(p,q)$ is \emph{trivial} if its components bound $n$ pairwise disjoint $2$-disks in $L(p,q)$.
% embedded pairwise disjoint $2$-disks $B^{2}_{1}, \ldots, B^{2}_{\nu}$ in $L(p,q)$.  
Two links $L,L'\subset L(p,q)$ are called \emph{equivalent} if there exists a smooth ambient isotopy between them.
% map $H\colon L(p,q) \times [0,1] \rightarrow L(p,q)$ where, if we define $h_{t}(x):=H(x,t)$, then $h_{0}=id_{M}$, $h_{1}(L)=L'$ and $h_{t}$ is a diffeomorphism of $L(p,q)$ for each $t \in [0,1]$.

% The setting of this paper is the \emph{Diff} category (of smooth manifolds and smooth maps). Every result also holds in the \emph{PL} category, and in the \emph{Top} category if we consider only tame links.

%Links can also be oriented, therefore throughout the thesis we will state each time whether it will be necessary a specification of the orientation or not.

\paragraph{Band diagrams and punctured disk diagrams}
Let $L$ be a link in $L(p,q)$ described by $p/q$-rational surgery over the unknot $U$ as before. 
Such a link can be presented as a mixed link $L \cup U \subset \s3$ \cite{DL}. Let $x$ be a point of $U$ and send it to $\infty$ in order to describe $\s3$ with the one-point compactification of $\mathbb{R}^{3}$. Let $(x_1, x_2, x_3)$ be the coordinates of $\R3$. Assume that $U$ is described by the $x_3$ axis. Consider the orthogonal projection $\mathbf{p}$ of $L\cup U$ on the $x_1x_2$ plane. Up to small isotopies of $L$, we can assume that this projection is \emph{regular}, that is:
\begin{enumerate}\itemsep-3pt
\item[1)] the projection of $L$ contains no cusps;
\item[2)] all auto-intersections of $\mathbf{p}(L)$ are transversal;
\item[3)] the set of multiple points is a finite set of double points.
\end{enumerate}

The \emph{punctured disk diagram} of $L \subset L(p,q)$ is a regular projection of $L\cup U$, where $U$ is projected to a "dot" (see Figure~\ref{pddbdiagram}).
If we just consider the space $\s3 \setminus U$ without the Dehn filling, a dotted diagram can also describe a link in the solid torus \cite{GM1}.

A \emph{band diagram} for a link $L \subset L(p,q)$ is obtained from a punctured disk diagram through the construction depicted in Figure \ref{pddbdiagram}: cut the punctured disk diagram along a line orthogonal to the boundary and avoiding the crossings of $L$, then deform the annulus into a rectangle.
This operation can easily be reversed.
%By reversing this operation we can obtain a punctured disk diagram from a band diagram of a link in a lens space.

\begin{figure}[h!]                      
\begin{center}                         
\includegraphics[width=12cm]{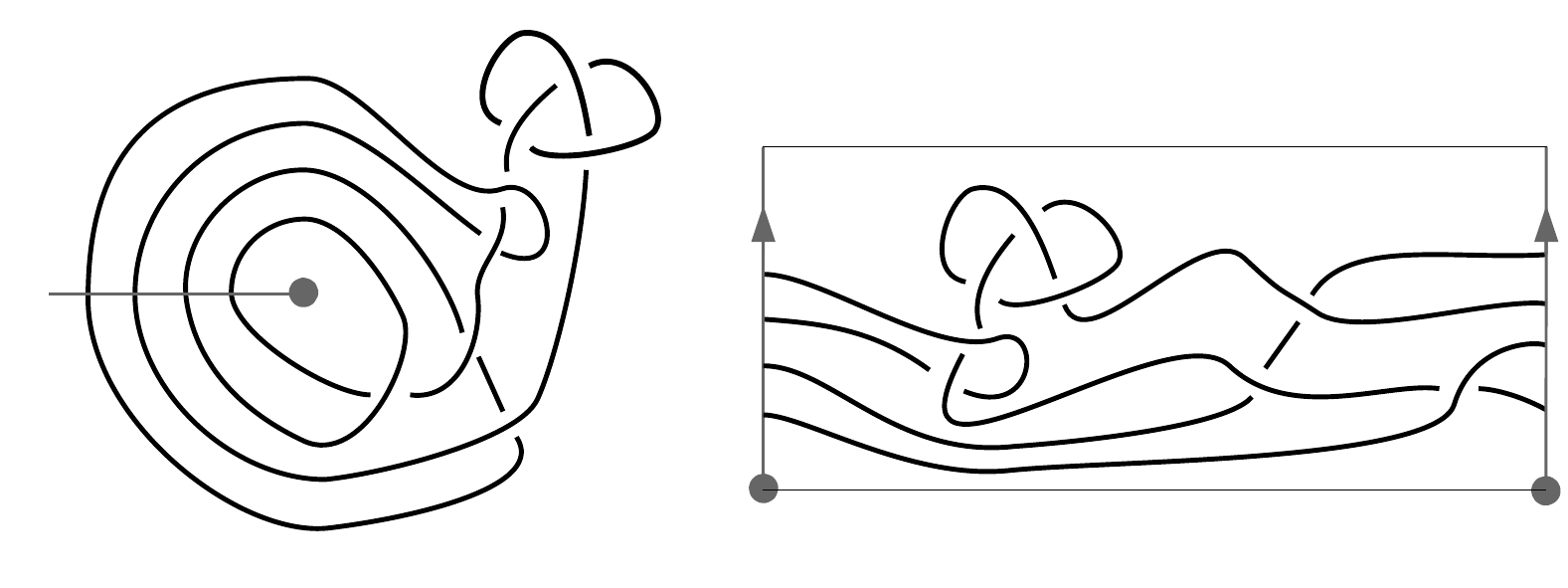}
\caption[legenda elenco figure]{An example of punctured disk diagram and band diagram for a link in $L(p,q)$.}\label{pddbdiagram}
\end{center}
\end{figure}

%\textcolor{red}{Bostjan please check if the SL move is correct. Which orientation choice are we performing?}

\paragraph{Reidemeister moves for punctured disk diagrams}
Punctured disk diagrams for $L(p,q)$ are accompanied by the three classical Reidemeister moves and an additional slide move $SL$ that arises from the $p/q$-surgery (see \cite{HP, Gb} for details).

%This moves arises from the rational surgery and is depicted in Figure~\ref{SL} (see \cite{Gb} for detailes).  

%Reidemeister-type moves for this diagrams are described in \cite{HP}, but is better to represent them on punctured disk diagrams as in \cite{Gb}. They work in the following way: besides the classical three Reidemeister moves, there is a fourth $SL$ move (slide move). This move acts in the following way: when a strand crosses the dot, a $(p,q)$-torus knot that contains all the diagram is added to that strand (see Figure~\ref{SL} for an example).

\begin{figure}[h!]                      
\begin{center}                         
\includegraphics[width=10cm]{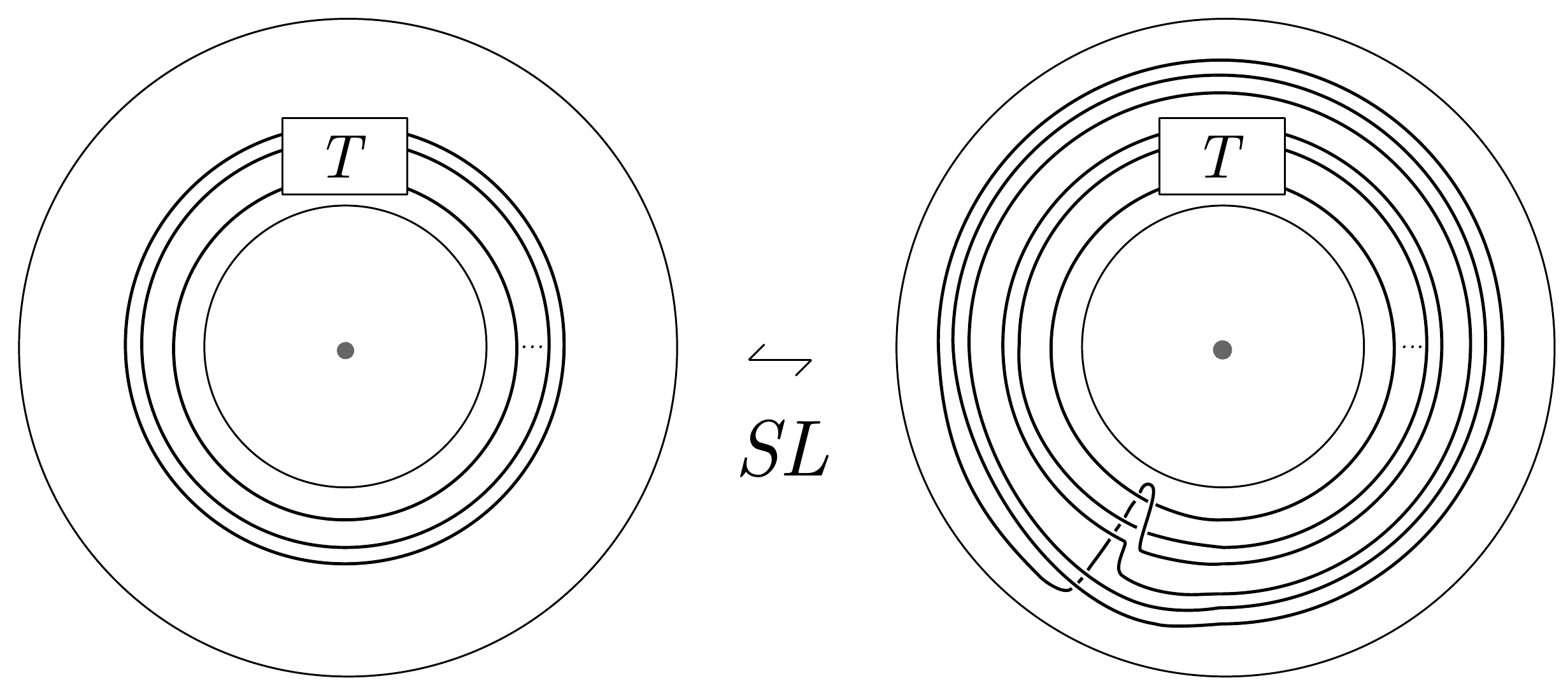}
\caption[legenda elenco figure]{$SL$ move on punctured disk diagram for the case $L(3,1)$.}\label{SL}
\end{center}
\end{figure}

\paragraph{The lens model for lens spaces}

%Let us assume that $p$ and $q$ are such that $\gcd(p,q)=1$ and $ 0 \leqslant q < p$. 

A lens space may be defined also by the following model: considering the $3$-dimensional ball $B^3$ and let $E_{+}$ and $E_{-}$ be respectively the upper and the lower closed hemisphere of $\partial B^{3}$. The equatorial disk $B^{2}_{0}$ is defined as the intersection of the plane $x_{3}=0$ with $B^{3}$. Label with $N$ and $S$ respectively the north pole $(0,0,1)$ and the south pole $(0,0,-1)$ of $B^{3}$.
Let \mbox{$g_{p,q} \colon E_{+} \rightarrow E_{+}$} be the counterclockwise rotation of $2 \pi q /p$ radians around the $x_{3}$-axis and let \hbox{$f_{3} \colon E_{+} \rightarrow E_{-}$} be the reflection with respect to the plane $x_{3}=0$ (Figure~\ref{L(p,q)}).

\begin{figure}[h!]                      
\begin{center}                         
\includegraphics[width=8cm]{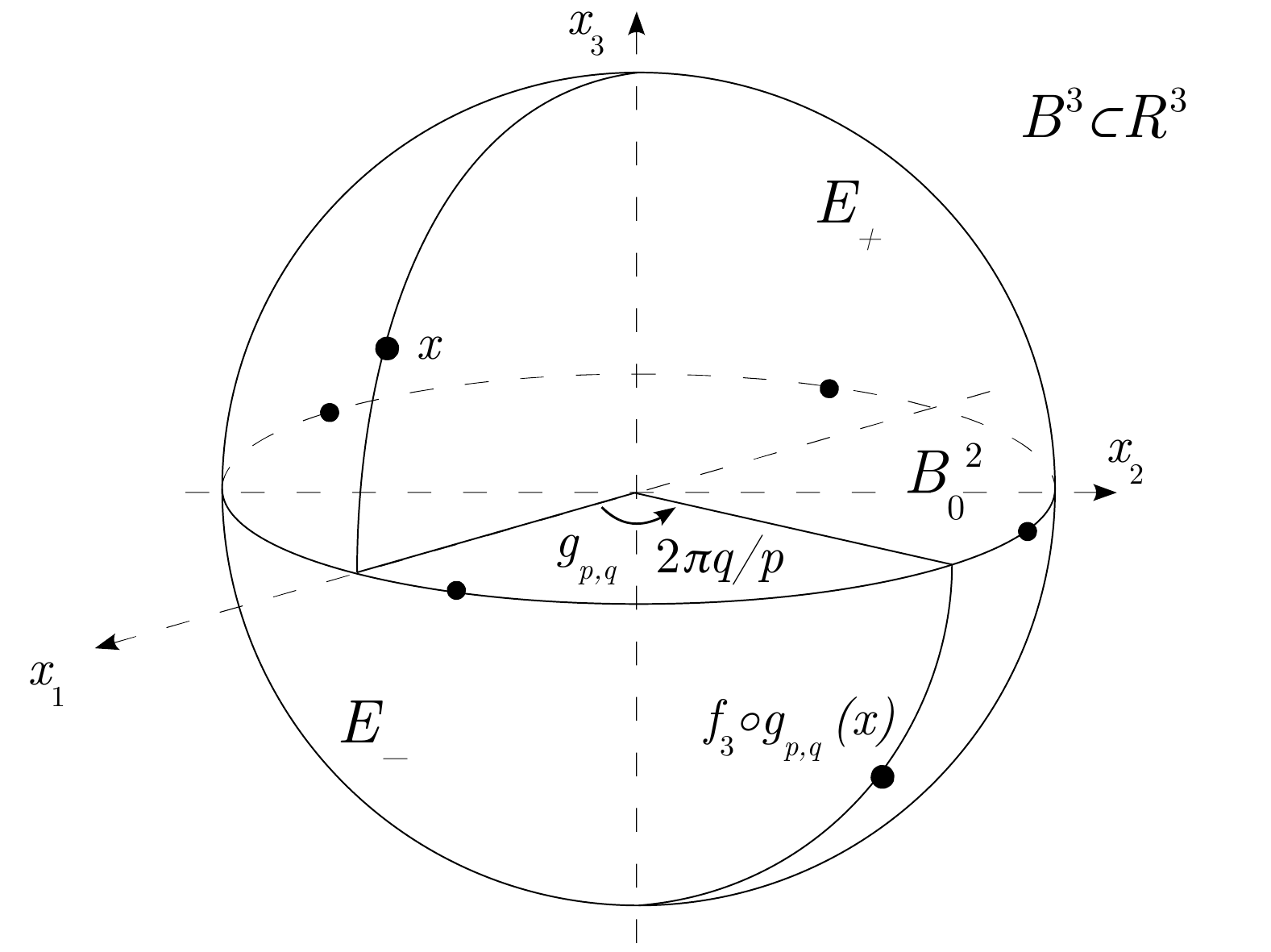}
\caption[legenda elenco figure]{Representation of $L(p,q)$.}\label{L(p,q)}
\end{center}
\end{figure}

The \emph{lens space} $L(p,q)$ is the quotient of $B^{3}$ by the equivalence relation on $\partial B^{3}$ which identifies $x \in E_{+}$ with $f_{3} \circ g_{p,q} (x) \in E_{-}$. The quotient map is denoted by \mbox{$F \colon B^{3} \rightarrow L(p,q)=B^{3} / \sim$}. Note that on the equator \mbox{$\partial B^{2}_{0}=E_{+} \cap E_{-}$} each equivalence class contains $p$ points, instead of the two points contained in the equivalence classes outside the equator. The first example is $L(1,0)\cong \s{3}$ and the second example is $L(2,1) \cong \rp{3}$, where the construction gives the usual model of the projective space with opposite points on $\partial B^3$ identified.

\paragraph{The disk diagram}
Since we are not interested in the case of $\s3$, we assume $p>1$. Intuitively, a \emph{disk diagrams} of a link $L$ in $L(p,q)$, represented by the lens model, is a regular projection of $F^{-1}(L)$ onto the equatorial disk of $B^{3}$, with the resolution of double points with overpasses and underpasses. 
In order to have a comprehensible diagram, we label with $+1, \ldots, +t$ the endpoints of the projection of the link coming from the upper hemisphere, and with $-1, \ldots, -t$ the endpoints coming from the lower hemisphere, respecting the rule $+i \sim -i$ (see \cite{CMM} for more details). 
An example is shown in Figure~\ref{link3}.

\begin{figure}[h!]                      
\begin{center}                         
\includegraphics[width=9cm]{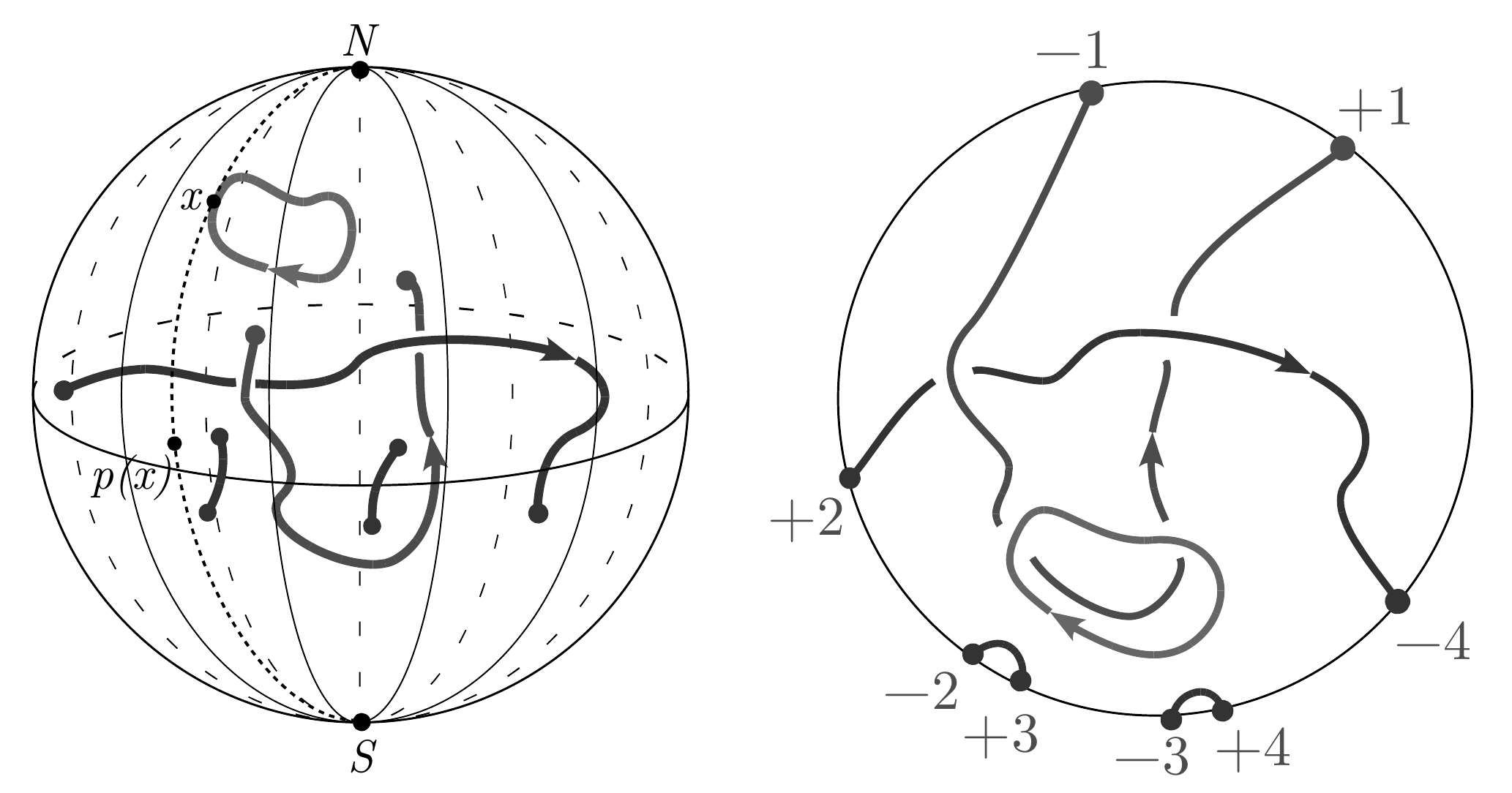}
\caption[legenda elenco figure]{A link in $L(9,1)$ and its corresponding disk diagram.}\label{link3}
\end{center}
\end{figure}

\paragraph{Reidemeister moves for disk diagrams}
%In \cite{CMM} it is shown that
Two disk diagrams of links in lens space represent equivalent links if and only if they are connected by a finite sequence of the Reidemeister moves illustrated in Figure \ref{R17} \cite{CMM}.

\begin{figure}[h!]                      
\begin{center}                         
\includegraphics[width=12cm]{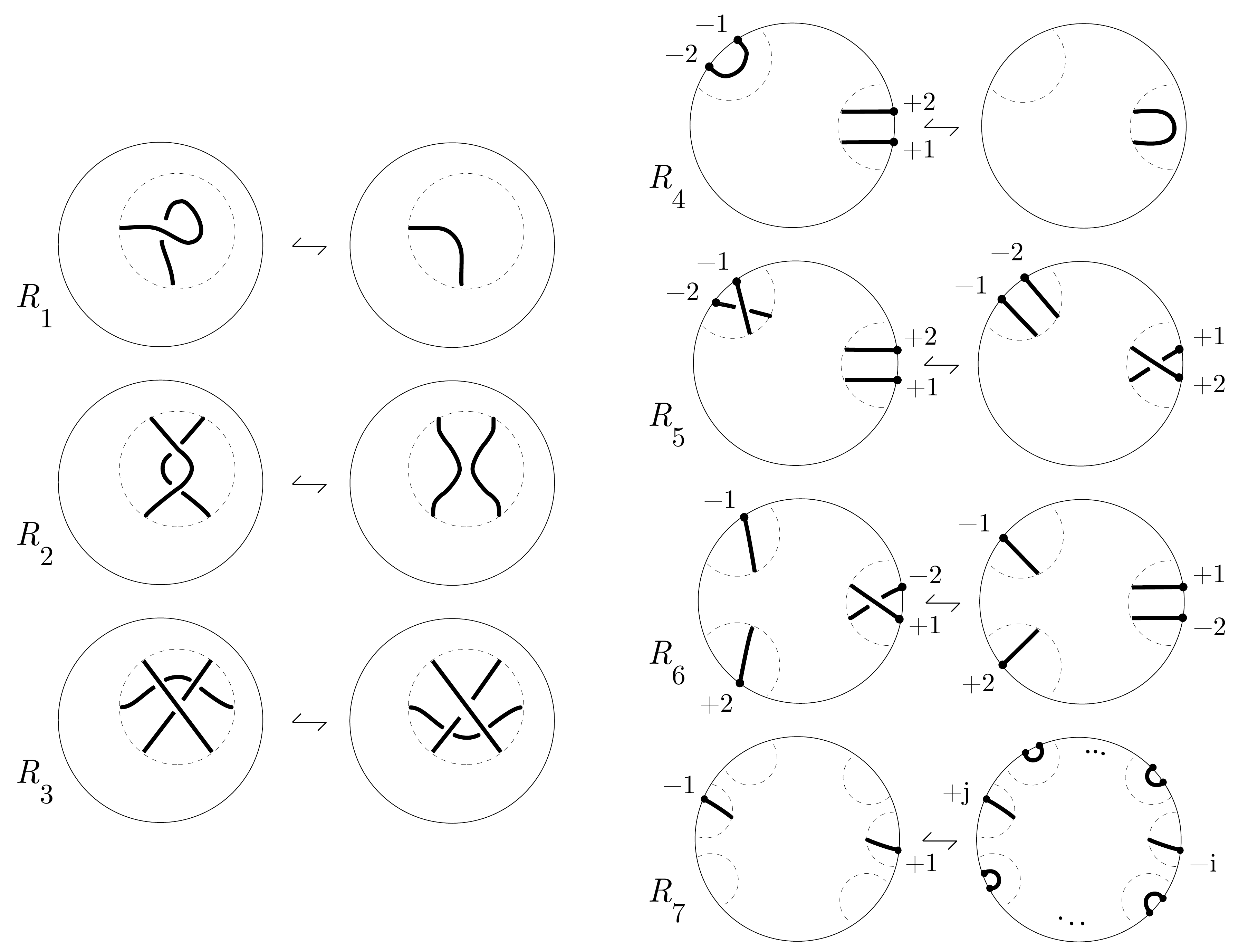}
\caption[legenda elenco figure]{Generalized Reidemeister moves.}\label{R17}
\end{center}
\end{figure}

\paragraph{Standard form of the disk diagram}
A disk diagram is defined \emph{standard} if the labels on its boundary points, read according to the orientation on $\partial B^{2}_{0}$, are $+1, \ldots, +t, -1, \ldots, -t$. 

\begin{prop}{\upshape \cite{Mn2}}\label{standard}
Every disk diagram can be reduced to a standard disk diagram using small isotopies:
if $p=2$, the signs of its boundary points can be exchanged;
if $p>2$, a finite sequence of $R_{6}$ moves can be applied in order to bring all positive boundary points aside.
\end{prop}

\paragraph{Equivalence between link diagrams}

At this point we describe a geometric transformation between disk and punctured disk diagrams.
Note that the transformation between punctured disk diagrams and band diagrams has already been described in the previous paragraphs.
%Since punctured disk and band diagrams are similar, we only show the connection between disk and band diagrams.

Let $B_{t}$ be the braid group on $t$ letters and let $\sigma_{1}, \ldots, \sigma_{t-1}$ be the Artin generators of $B_{t}$. Consider the Garside braid $\Delta_{t}$ on $t$ strands defined by 
$$ (\sigma_{t-1}\sigma_{t-2} \cdots \sigma_{1})( \sigma_{t-1}\sigma_{t-2} \cdots \sigma_{2}) \cdots (\sigma_{t-1}), $$
 and illustrated in Figure~\ref{treccia}. Note that $\Delta_{t}$ represents a positive half-twist of all the braid strands. The braid $\Delta_{t}^{2}$ belongs to the center of the braid group, i.e. it commutes with every braid. Moreover, $\Delta_{t}^{-1}$ can be written, after some braid operations, as 
$$ (\sigma_{t-1}^{-1}\sigma_{t-2}^{-1} \cdots \sigma_{1}^{-1})( \sigma_{t-1}^{-1}\sigma_{t-2}^{-1} \cdots \sigma_{2}^{-1}) \cdots (\sigma_{t-1}^{-1}). $$

\begin{figure}[h!]                      
\begin{center}                         
\includegraphics[width=11cm]{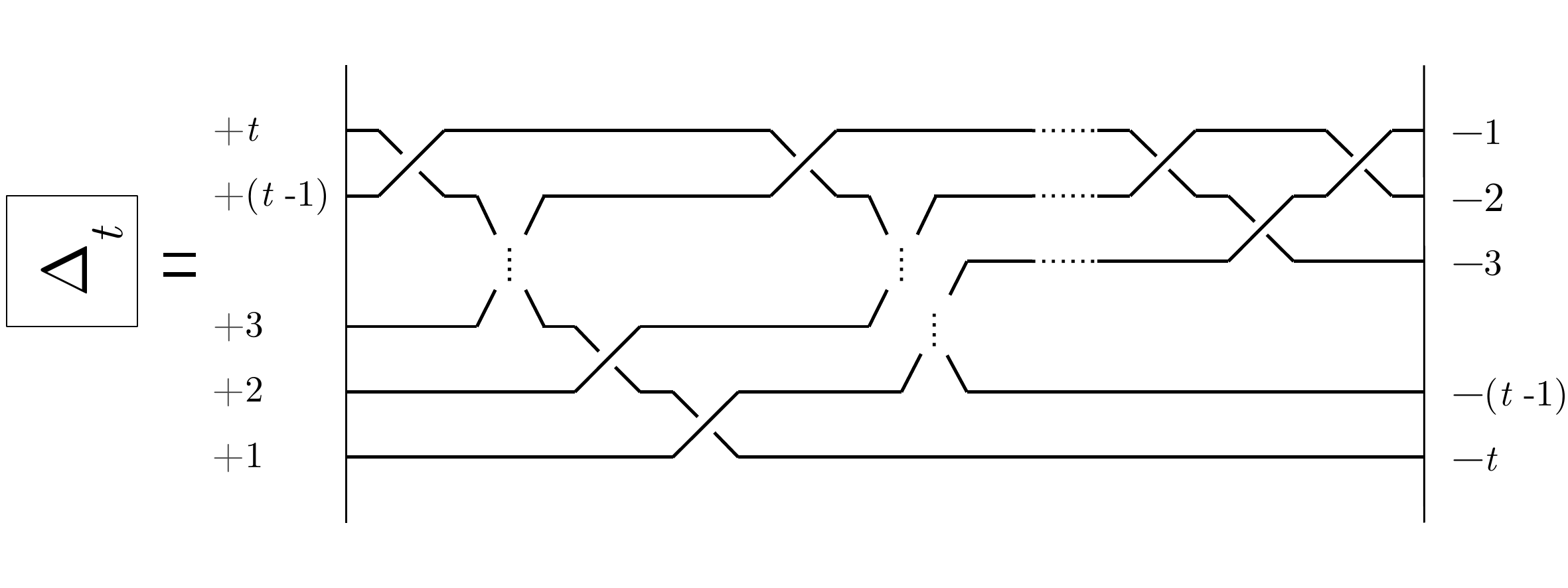}
\caption[legenda elenco figure]{The braid $\Delta_{t}$.}\label{treccia}
\end{center}
\end{figure}

The following Proposition explains how to transform a band diagram into a standard disk diagram. 

\begin{prop}\label{BL}
Let $L$ be a link in $L(p,q)$ assigned via a band diagram $B_{L}$. A standard disk diagram $D_{L}$ representing $L$ can be obtained by the construction depicted in Figure~\ref{B-L}.

\begin{figure}[h!]                      
\begin{center}                         
\includegraphics[width=10cm]{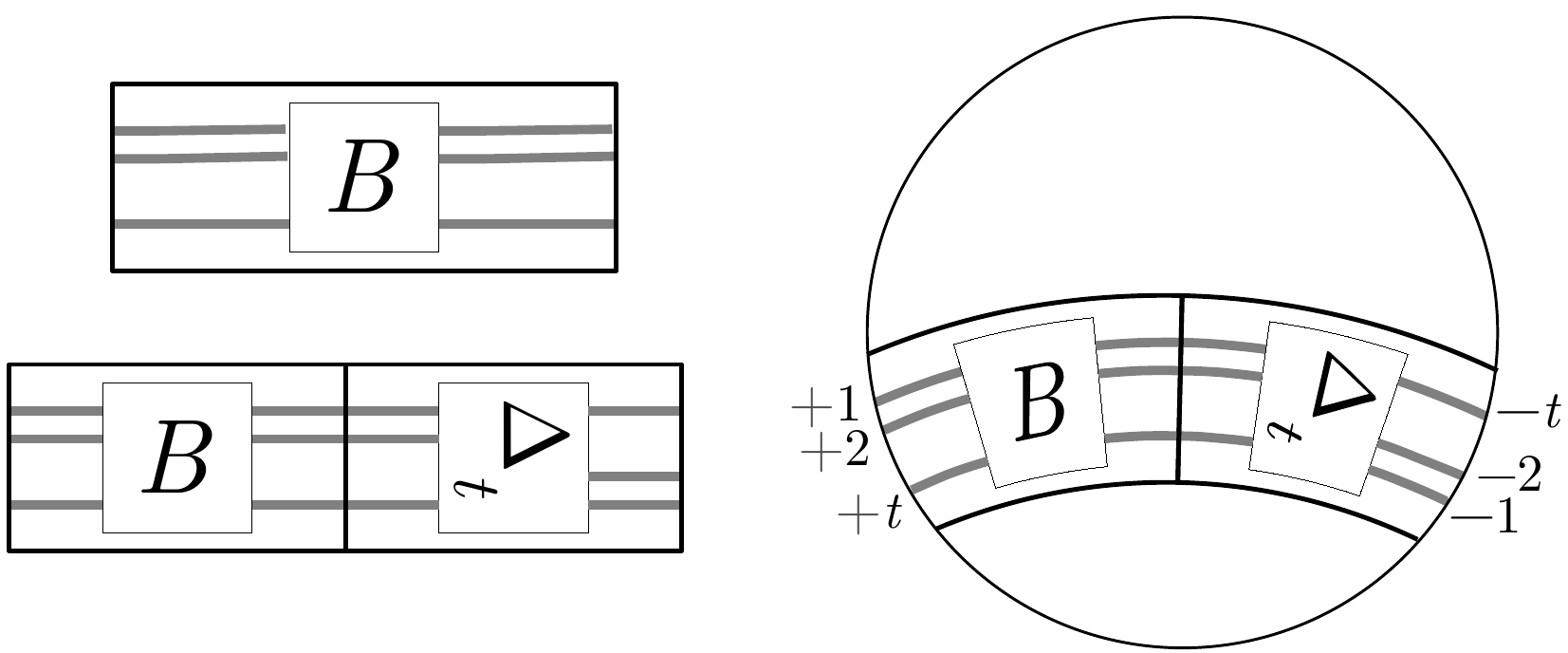}
\caption[legenda elenco figure]{From band diagram $B_{L}$ to disk diagram $D_{L}$ in $L(p,q)$.}\label{B-L}
\end{center}
\end{figure}

Consider the band diagram $B_{L}$, the rectangle has two opposite identified sides, with $t$ points on each of them; add to the right side of the band diagram the braid $\Delta_{t}$, then put the resulting band inside a disk, with the opposite sides of the new rectangle on the boundary of the disk. Add the indexation $+1, +2, \ldots, +t$ on the points of the left side of the rectangle and $-1, -2, \ldots, -t$ on the other boundary points: the result is the desired disk diagram $D_{L}$.
\end{prop}

\begin{proof}
The band diagram may be seen as the result of a genus one Heegaard splitting of $L(p,q)$, where the link is contained inside one of the two solid tori and is regularly projected on the annulus which has as a boundary of $q$ longitudes of the solid torus. Following the geometric description of the equivalence between the Heegaard splitting model and the lens model of the lens spaces, depicted for the particular case of $L(5,2)$ in Figure~\ref{QtoHS}, 
\begin{figure}[h!]                      
\begin{center}                         
\includegraphics[width=9.8cm]{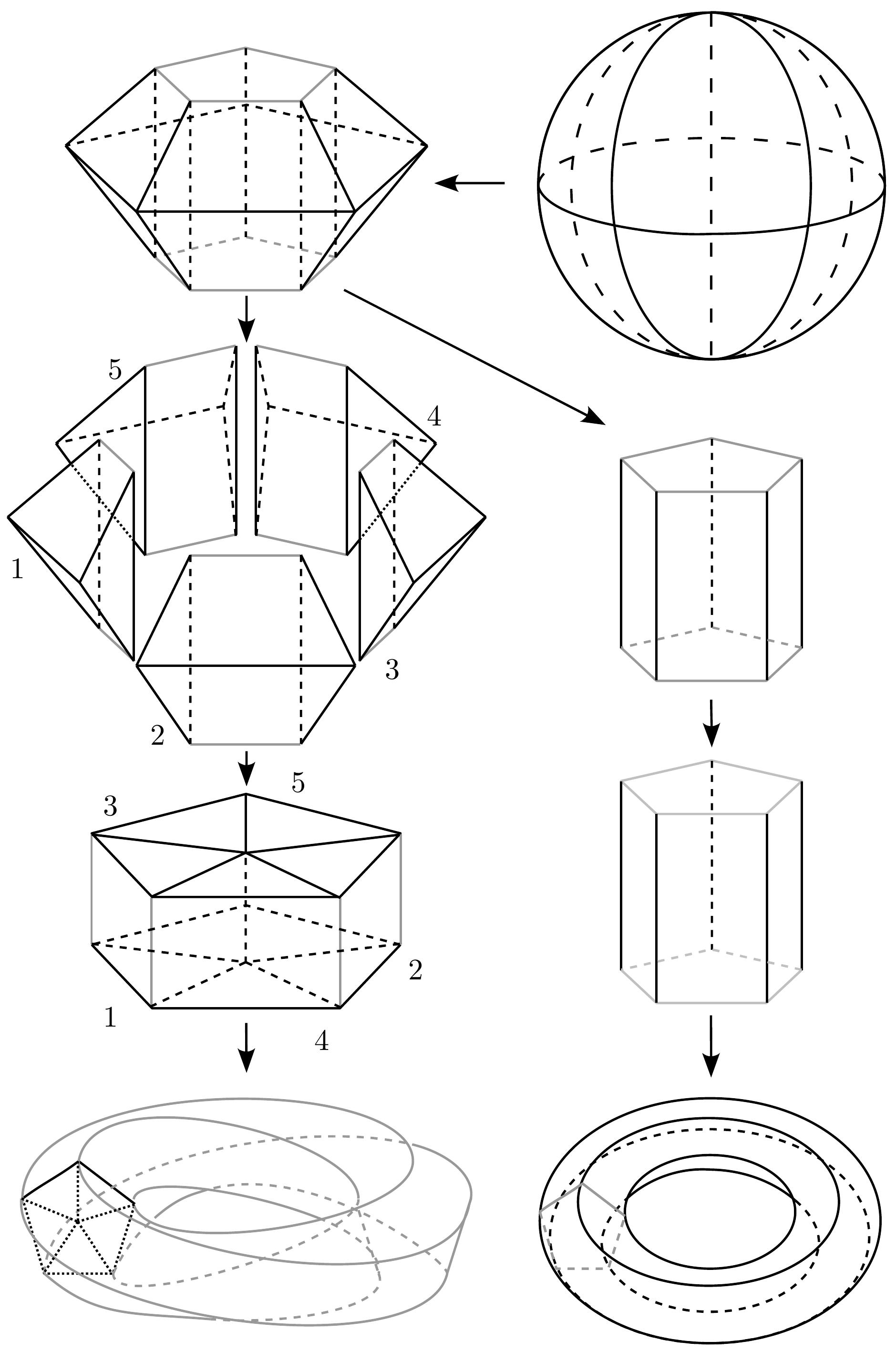}
\caption[legenda elenco figure]{From the lens model to Heegaard splitting.}\label{QtoHS}
\end{center}
\end{figure}
we can put the band diagram in one solid torus as depicted in Figure~\ref{B-Ldim}, then put the solid torus inside the lens model of the lens space, and project the band diagram onto the equatorial disk. During this operation, we have a twist, described by $\Delta_{t}$. Finally, adding the labels to the boundary points, we get the desired disk diagram $D_{L}$.
\begin{figure}[h!]                      
\begin{center}                         
\includegraphics[width=11cm]{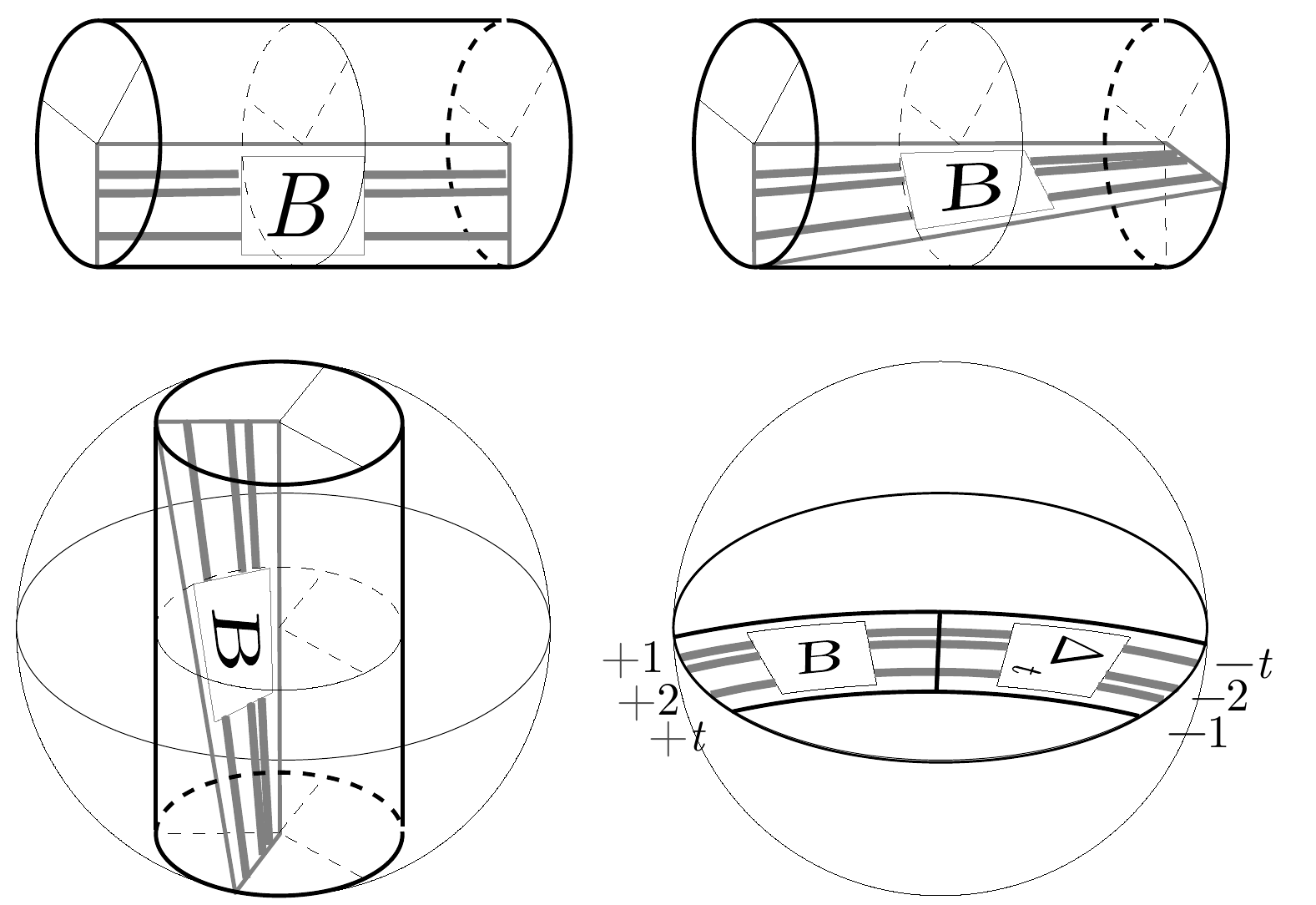}
\caption[legenda elenco figure]{From the Heegaard splitting to the lens model of $L(p,q)$.}\label{B-Ldim}
\end{center}
\end{figure}
\end{proof}

On the other hand, we can recover the band diagram of a link from the disk diagram.

\begin{prop}\label{DB}
Let $L$ be a link in $L(p,q)$, described by a disk diagram. Let $D_{L}$ be the standard disk diagram obtained from Proposition \ref{standard}. A band diagram $B_{L}$ for $L$ is  constructed using the geometric algorithm described in Figure~\ref{L-B}.

\begin{figure}[h!]                      
\begin{center}                         
\includegraphics[width=10cm]{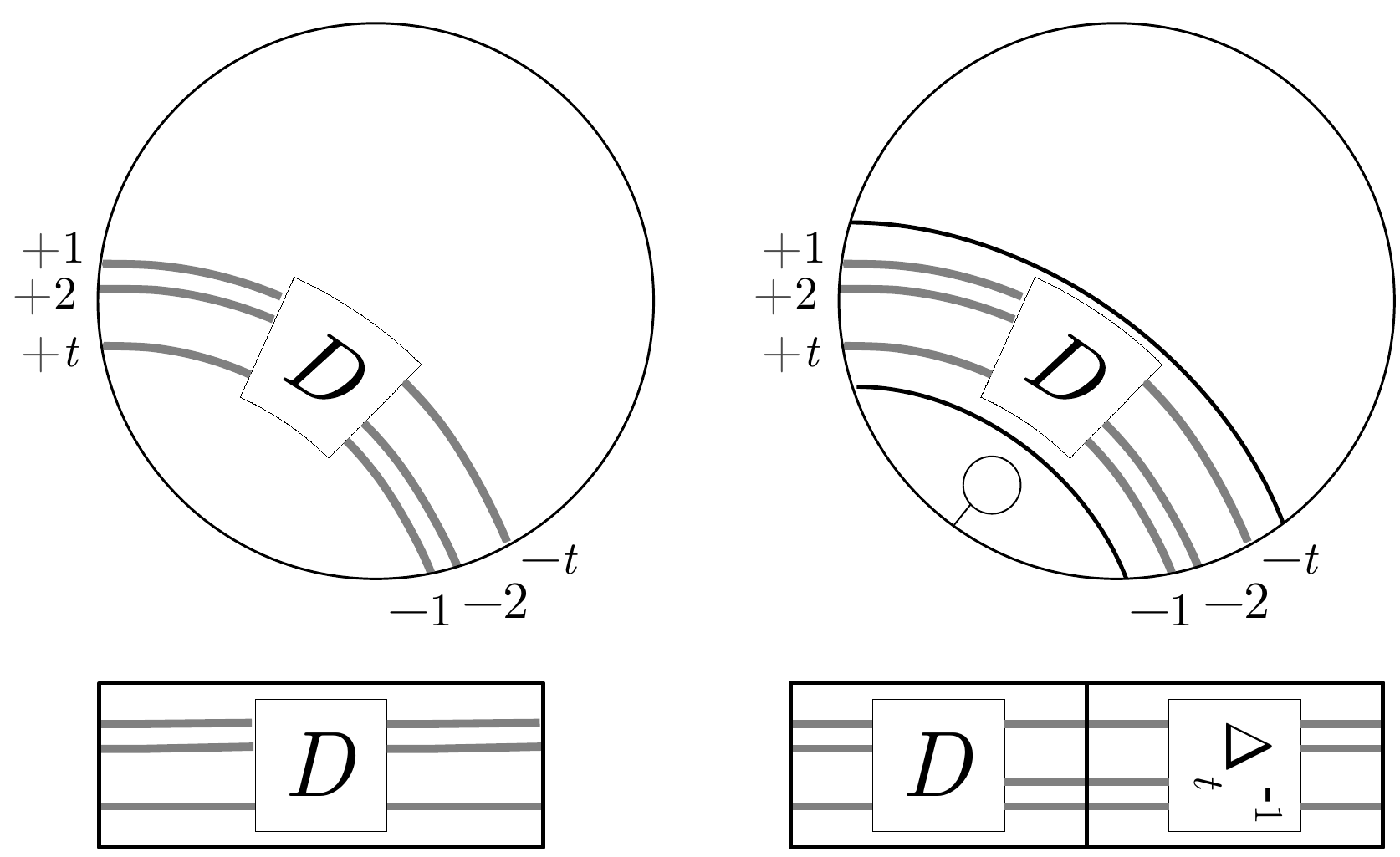}
\caption[legenda elenco figure]{From disk diagram $D_{L}$ to band diagram $B_{L}$ in $L(p,q)$.}\label{L-B}
\end{center}
\end{figure}

Consider the disk diagram $D_{L}$ and open the disk on the right of the $+1$ point, as in Figure~\ref{L-B}; this way a rectangle is obtained, with identified points only on the left and right sides, then add the braid $\Delta_{t}^{-1}$ on the right side and this is the desired band diagram for $L$. 
\end{prop}

\begin{proof}
It is exactly the converse geometric construction of the proof of Proposition~\ref{BL}. \end{proof}

The intuitive correspondence between the Reidemeister moves on the disk diagrams and the band diagrams are represented in Table~\ref{movesB}.

\begin{table}[h!]
\label{movesB}
\begin{center}
\begin{tabular}{|c|c|}
\hline 
Disk diagram & Band diagram  \\
\hline
$R_1$ & $R_1$ \\
\hline
$R_2$ & $R_2$  \\
\hline
$R_3$ & $R_3$   \\
\hline
$R_4$ & isotopy of an arc  \\
\hline
$R_5$ & isotopy of a crossing  \\
\hline
$R_6$ & not allowed on standard diagram   \\
\hline
$R_7$ & $SL$  \\
\hline
\end{tabular}
\end{center}
\label{table}
\end{table}

\section{The KBSM of links in $L(p,q)$ is an essential invariant}

In this section we provide examples of different links with equivalent lifts of \cite{Mn2} to show that the KBSM is an essential invariant, that is to say, it may distinguish links with equivalent lifts.

\paragraph{KBSM of links in lens spaces via punctured disk diagrams}
The KBSM of a 3-manifold M, denoted by $S_{2, \infty}(M)$, is defined in the following way. 
Take an unoriented framed link $L \subset M$ and let $\mathcal{L}_{\textit{fr}}$ be the set of ambient isotopy classes of unoriented framed links in $M$, where we also add the empty knot $\emptyset$. Let $L^{(n)}$ denote the framed link obtained by $L \subset \mathcal{L}_{\textit{fr}}$ by adding $n$ full right-handed twists to the framing. Let $R=\mathbb{Z}[A^{\pm1}]$ be the ring of Laurent polynomials in variable $A$.
Define $\mathcal{S}_{\textit{fr}}$ to be the submodule of $R \mathcal{L}_{\textit{fr}}$ generated by the skein relations $L-AL_{0}-A^{-1}L_{\infty}$ and $L^{(1)}+A^{3}L$, where $L_{0}$ and $L_{\infty}$ denote the links obtained by the resolutions of one crossing of $L$ as Figure~\ref{skeincrossing} shows.

\begin{figure}[h!]                      
\begin{center}                         
\includegraphics[width=8cm]{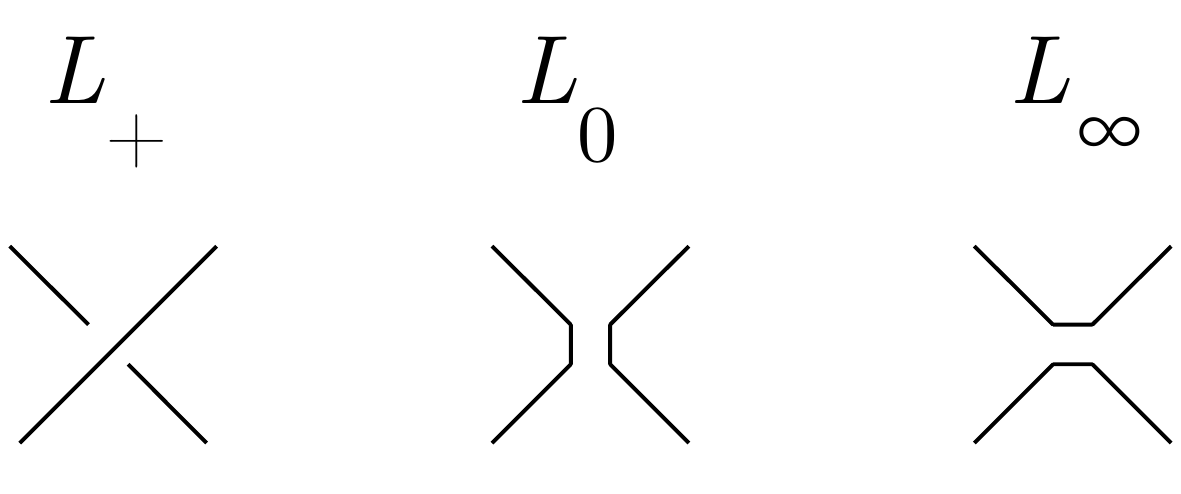}
\caption[legenda elenco figure]{Resolution of a crossing of $L$.}\label{skeincrossing}
\end{center}
\end{figure}

The \emph{Kauffman bracket skein module} is the quotient $S_{2,\infty}(L(p,q))=R \mathcal{L}_{\textit{fr}} / \mathcal{S}_{\textit{fr}} $.

In order to understand the skein module for a particular 3-manifold $M$, we have to present a basis of the module and understand the torsion (if it exists).
%If we want to understand this skein module, we have to find a basis of it and understand the torsion if it exists.
%A similar definition of the KBSM can be given for any oriented $3$-manifold, and it is interesting for us to recall the KBSM of the solid torus.

We use the representation of links in lens space given by a punctured disk diagram. These diagrams are useful also to represent links in the solid torus (see Section \ref{bandsection}). Let $x_{0}$ denote the local unknot in the solid torus and $x_{i}$ denote the link with $i$ components described in Figure~\ref{toricbasis}. 

\begin{figure}[h!]                      
\begin{center}                         
\includegraphics[width=12cm]{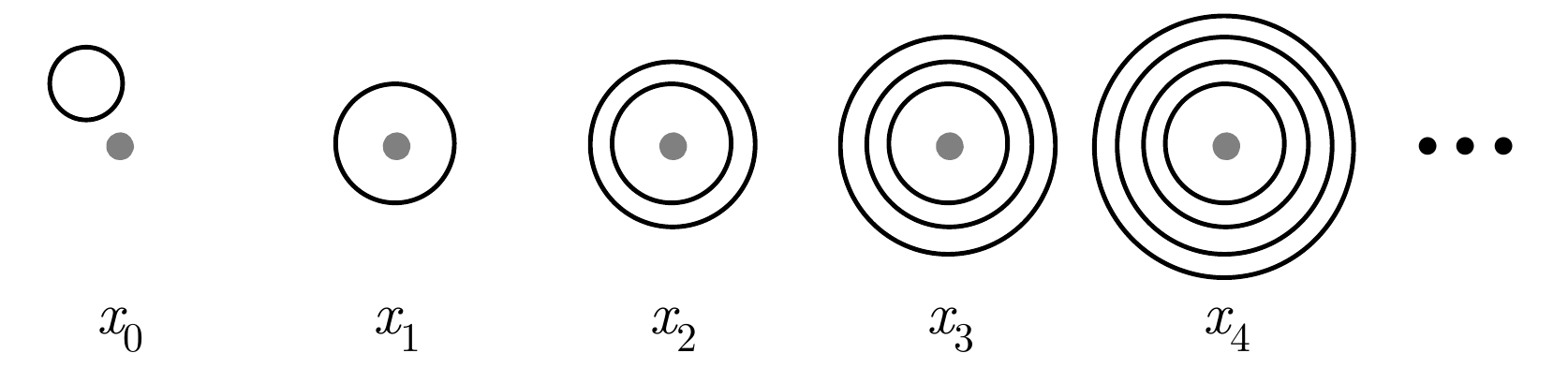}
\caption[legenda elenco figure]{KBSM basis for the punctured disk diagram.}\label{toricbasis}
\end{center}
\end{figure}

\begin{prop}{\upshape \cite[Corollary 2]{Tu1}}\label{solidtorus}
The KBSM of the solid torus is freely generated by the set $\{ x_{i} \}_{i \in \mathbb{N}}$.
\end{prop}

The following proposition has been presented in \cite{HP}:

\begin{prop}{\upshape \cite[Theorem 4]{HP}}\label{KBSMlpq}
For $p\geq 1$ the KBSM of $L(p,q)$ is freely generated by $x_{0}, x_{1}, \ldots, x_{\lfloor p/2 \rfloor}$, where $\lfloor r \rfloor$ denotes the integer part of $r$.
\end{prop}

\begin{oss}
The computation of the Kauffman bracket of a link in $L(p,q)$ described by a punctured disk diagram is performed using the following algorithm: simplify all the crossings with the skein relation and once you are left only with the diagrams of Figure \ref{toricbasis}, substitute each $x_{i}$ for all $i > \lfloor p/2 \rfloor$ with a suitable linear combination of the basis $x_{0}, x_{1}, \ldots, x_{\lfloor p/2 \rfloor}$.
The formula for $x_i$, $i > \lfloor p/2 \rfloor$, can be found by considering $x_{i-p}$ (or $x_{p-i}$ if $i-p<0$), applying an $SL$ move and resolving the crossings with the skein relation. In the case where the winding number $w(L)$ of the link in the lens space is lower than $\lfloor p/2 \rfloor$, the KBSM of links in lens spaces coincide with the one in the solid torus.
%\textcolor{red}{Shall we put a computational example, maybe for  x2 or x3 in L(3,1)?}
\end{oss}

As a consequence, the KBSM of a knot in $L(p,q)$ can be recovered from the KBSM of the corresponding knot in the solid torus $T$ under the standard inclusion of $T \subset L(p,q)$. Through this method \cite{Gb} provided the KBSM-s of knots in $L(p,q)$ up to $5$ crossings.

%\textcolor{red}{For $L(2,1)$, which is the substitution for $x_{1}$ so that the KBSM is equivalent to Drobotukhina Jones invariant? Basta scegliere x1=delta ma con un cambio di base...}

%\textcolor{red}{For $L(p,q)$, there is a substitution for the generators so that the KBSM is equivalent to BHMV invariants?}

%===========================

\paragraph{Lift of links in lens spaces}

Following \cite{Mn2} we are able to construct a diagram of the lift of a link in $L(p,q)$ starting from a disk diagram. Using Proposition \ref{DB} we are able to construct a similar diagram of the lift starting from a band diagram, as shown in the following proposition.

\begin{prop}\label{teoliftband}
Let $L$ be a link in the lens space $L(p,q)$ and let $B$ be a band diagram for $L$ with $t$ boundary points; then a diagram for the lift $\widetilde L \subset \s3$ can be found by juxtaposing $p$ copies of $B$ and closing them with the braid $\Delta_{t}^{2q}$ (see Figure~\ref{bandlift}).

\begin{figure}[h!]                      
\begin{center}                         
\includegraphics[width=12.5cm]{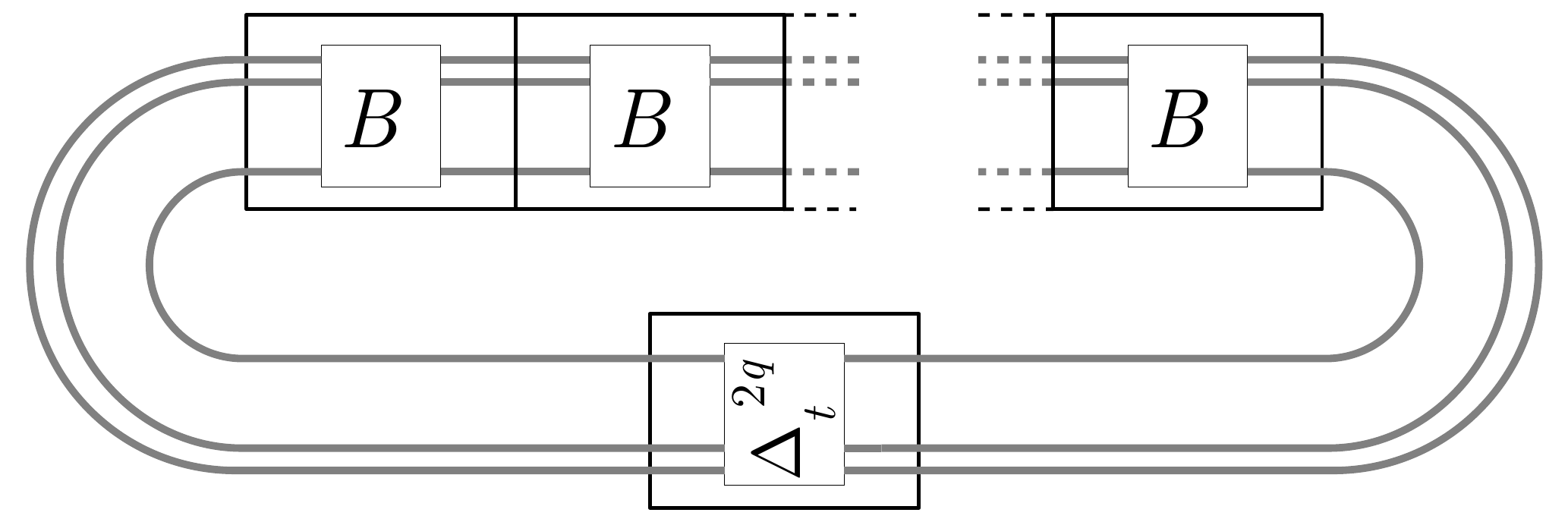}
\caption[legenda elenco figure]{Diagram of the lift of a link in lens spaces from its band diagram.}\label{bandlift}
\end{center}
\end{figure}
\end{prop}

\begin{proof}
Consider the planar diagram of the lift of \cite[Theorem 3]{Mn2} and convert the standard disk diagram $D_{L}$ plus the braid $\Delta_{t}^{-1}$ to the equivalent band diagram $B_{L}$. This gives exactly the diagram for $\widetilde{L}$ illustrated in Figure~\ref{bandlift}.
\end{proof}

\begin{oss}\label{Chbili}
The lift in $\s3$ of a link $L\subset L(p,q)$ is exactly a \emph{$(p,q)$-lens link} in $\s3$, according to \cite{C4}. Precisely, the $n$-tangle $T$ that Chbili uses in his construction is the band diagram $B_{L}$ for $L$.
In the same paper he makes explicit that the lift is a freely periodic link in $\s3$.
\end{oss}

\paragraph{Essential invariants}

It is clear that every link invariant in $\s3$ induces a link invariant in $L(p,q)$ if the first invariant is computed on the lift. % This allows us to compute a lot of invariants.
On the contrary, it is important to know if an link invariant for $L(p,q)$ is a real $L(p,q)$ invariant or just an $S^3$ invariant in disguise. If an $L(p,q)$ invariant can take different values on two links with equivalent lifts, we call such an invariant \emph{essential}. 

%links in lens spaces is not only an invariant of the lift; we will call \emph{essential} this kind of invariant.

Inequivalent links with equivalent lifts are perfect candidates to check whether an invariant $I$ of links in lens spaces is essential: find two inequivalent knots $K_1$ and $K_2$ with equivalent lifts such that $I(K_1) \neq I(K_{2})$. %In order to compute the KBSM on the links described in \cite{Mn2} with a disk diagram, we have to get a punctured disk diagram for them.

\paragraph{KBSM is an essential invariant}

In \cite{Mn2} are shown several examples which consist of different links with equivalent lifts. By applying Proposition \ref{DB} we can transform the disk diagram of the knots into band diagrams and consequently compute the Kauffman bracket for them. The result of the computations shows us that the KBSM is an essential invariant.

\begin{ese}\label{es1KBSM}

In Figure~\ref{CE1band} are represented the punctured disk diagrams of the knots  $K_1$ and $K_2$ in $L \left (p, \frac{p \pm 1}{2} \right )$. When $p>3$ and odd, the two knots are not isotopic and they both lift to the unknot in $\s3$, as shown in \cite[Example 9]{Mn2}. 
\begin{figure}[h!]                      
\begin{center}                       
\includegraphics[width=8cm]{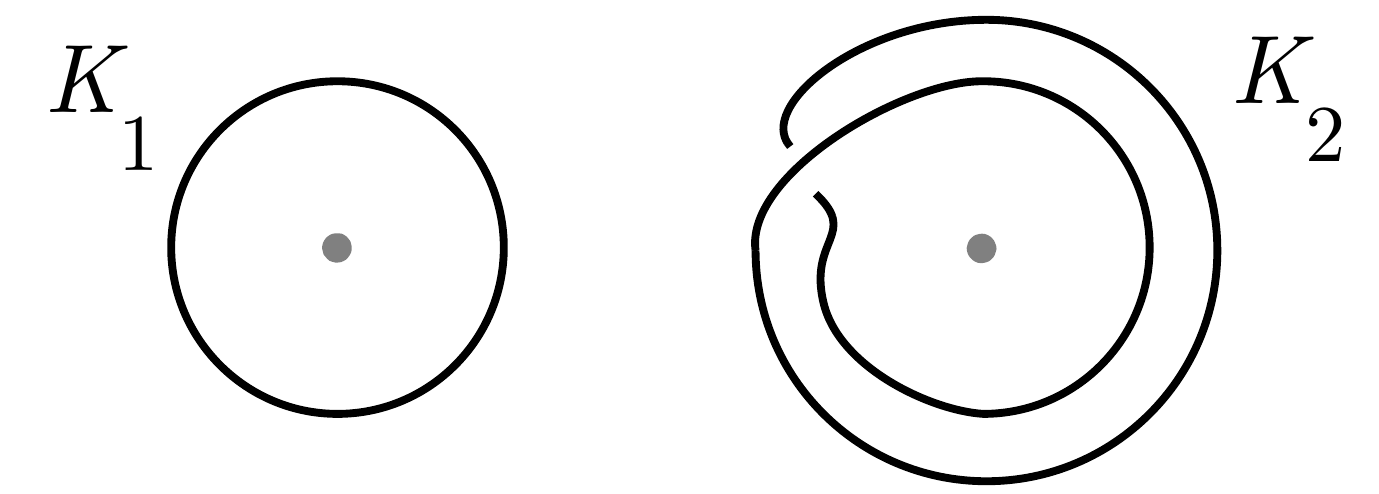}
\caption[legenda elenco figure]{Punctured disk diagrams $K_1$ and $K_2$ in $L \left (p, \frac{p \pm 1}{2} \right )$.}\label{CE1band}
\end{center}
\end{figure}
It holds that $\textrm{KBSM} (K_{1})=x_{1}$ and $\textrm{KBSM} (K_{2})=Ax_{2}+A^{-1}x_{0}$.
\end{ese}

\begin{ese}\label{es2KBSM}

The two links $L_A$ and $L_B$ in $L(4,1)$, represented by the punctured disk diagrams of Figure~\ref{CE2band}, are not isotopic since they have a different number of components, but they both lift to the Hopf link, as shown in \cite[Example~10]{Mn2}.
 \begin{figure}[h!]                      
\begin{center}                         
\includegraphics[width=8cm]{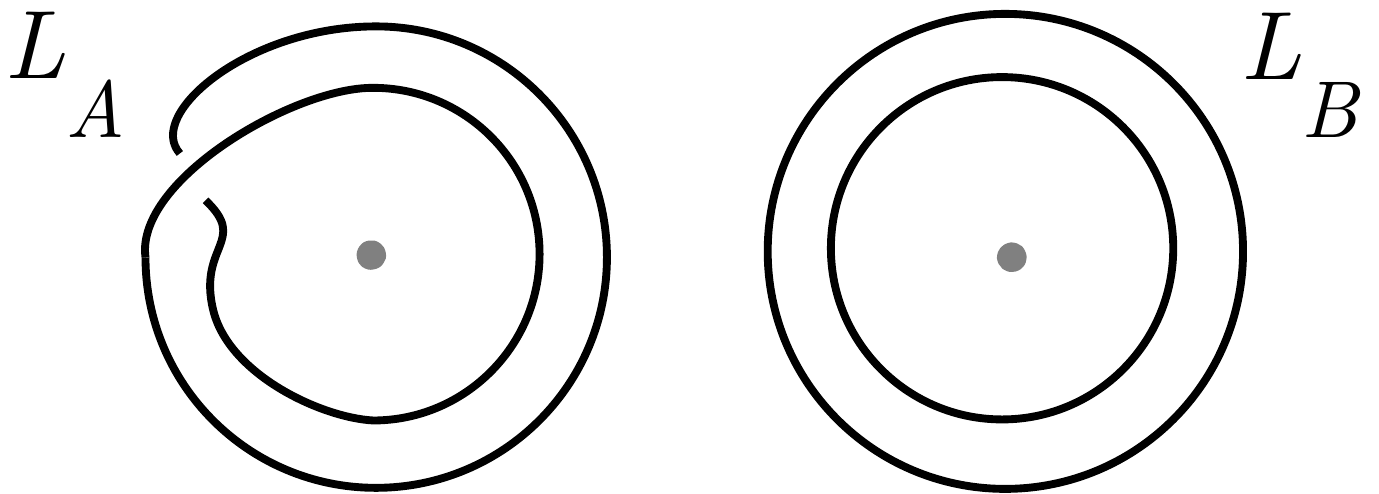}
\caption[legenda elenco figure]{Punctured disk diagrams for $L_A$ and $L_B$ in $L(4,1)$.}\label{CE2band}
\end{center}
\end{figure}
It holds that $\textrm{KBSM} (L_{A})=Ax_{2}+A^{-1}x_{0}$ and $\textrm{KBSM} (L_{B})=x_{2}$.
\end{ese}

\begin{ese}
Consider the two links $A_{2,2}$ and $B_{2,2}$ in $L(4,1)$, represented by the punctured disk diagrams illustrated in Figure~\ref{CE3band}. Their punctured disk diagram are found according to Proposition~\ref{DB}, starting from \cite[Example 11]{Mn2}. The links are not isotopic as they have different Alexander polynomials, but they have equivalent lifts.
\begin{figure}[h!]       
\center                         
\includegraphics[width=10cm]{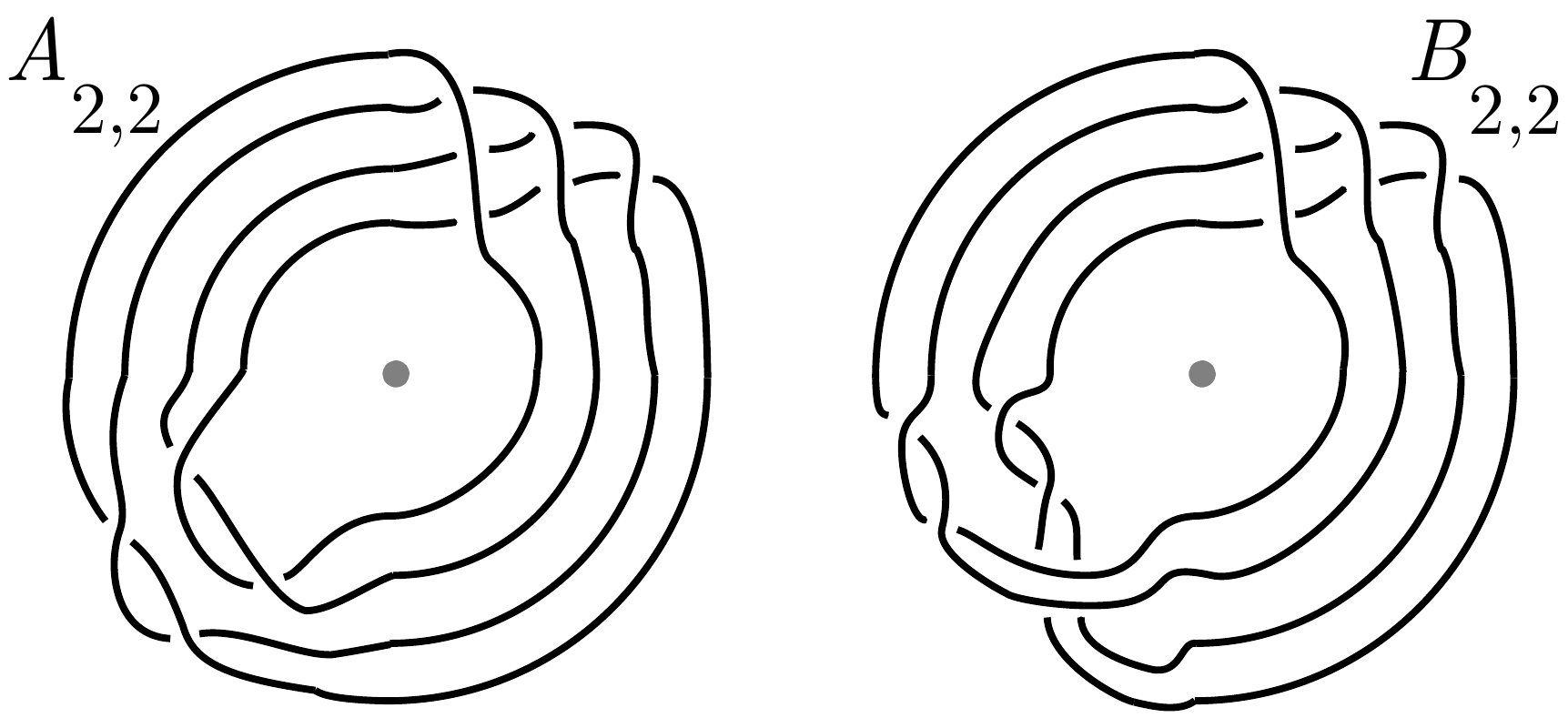}
\caption[legenda elenco figure]{Punctured disk diagrams for $A_{2,2}$ and $B_{2,2}$ in $L(4,1)$.}\label{CE3band}
\end{figure}

%The skein reduction tree is quite big, therefore we report here only the final result:
Their KBSM-s are again different:
\begin{multline*}
\textrm{KBSM}  
(A_{2,2})= (3 A^{14} + A^{10} -2A^{8}-A^{6}-1) x_{2} + \\
+ ( -A^{20} + A^{12} -A^{6} -A^{-2} +A^{-4} -A^{-8} +A^{-12}-A^{-16}) x_0 
\end{multline*}
\begin{multline*}
\textrm{KBSM}  (B_{2,2})
=  (3 A^{12} + A^{8} -2A^{6}+2A^{2}+2+4A^{-2}+2A^{-4}) x_{2} + \\
+ ( -A^{18} + A^{10} -A^{4} +1 +3A^{-4} -A^{-6} +5 A^{-8} -3A^{-10}+4A^{-12}-3A^{-14} -A^{-18}) x_0. 
\end{multline*}

\end{ese}

%================================

\section{KBSM vs KB of the lift}

In this section we show a relation between the KBSM of a link $L$ in $L(p,q)$ and the Kauffman bracket of the lift.

%In this section we show a relation between the KBSM of a link $L$ in the solid torus and the Kauffman bracket of the lift of the corresponding link in lens spaces.

Let $B$ be a punctured disk diagram of $L$ (or equivalently, a band diagram, as Section~\ref{bandsection} shows).
We can consider the punctured disk diagram as a diagram of $L$ in the solid torus, and compute the KBSM of $L$ inside the solid torus according to Proposition~\ref{solidtorus}. 
%With some suitable substitutions of the values $x_{i}$ for $i> \lfloor p/2 \rfloor$, the KBSM of the link in $L(p,q)$ can be found. 
Let $\delta=-A^{-2}-A^{2}$. A substitution of $x_0$ with $1$ and $x_i$, $i>0$, with $\delta^{i-1}$ yields the usual Kauffman bracket of links in $\s3=L(1,0)$ which we denote by $\langle L \rangle$.

%If we substitute $x_0$ with $1$ and every $x_i$ with $\delta^{i-1}$ when $i>0$, we are computing the usual Kauffman bracket of links in $\s3=L(1,0)$, and we denote it with the usual bracket $\langle L \rangle$.

Let $t$ denote half of the number of boundary points of $B$. Following Remark~\ref{Chbili} we can describe $B$ as a $t$-tangle, and we can describe the lift of the link as a sum of $t$-tangles: $B_{\phantom{t}}^{p}\Delta_{t}^{2q}$. The closure of a $t$-tangle $\widehat{B}$ is a link in $\s3$ and it coincides with the link in $\s3=L(1,0)$ represented by the band diagram $B$. In the case where the winding number $w(L)$ of the link in the lens space is lower than $\lfloor p/2 \rfloor$, the statement is true also for the KBSM of the link in the lens space.
Let $Id_{t}$ denote the trivial braid on $t$ strands, or equivalently, the trivial $t$-tangle. The link $\widehat{Id_{k}}$ is the unlink with $k$ components.

\begin{teo}
Let $L$ be a link in $L(p,q)$ described by the band diagram $B$. Then 
\begin{equation}\label{KBSMKB}
\langle \widehat{B} \rangle^{p} \equiv \langle \widehat{B_{\phantom{t}}^{p}\Delta_{t}^{2q}} \rangle \bmod I,
\end{equation}
where $I$ is the ideal generated by $p$, $\delta^{p-1}-1$ and the family $ \{ \langle \widehat{Id_{t-2i}} \rangle^{p} - \langle \widehat{\Delta_{t-2i}^{2q}} \rangle  \}_{i=0, \ldots, \lfloor (t-1)/2 \rfloor }$.
\end{teo}

Notice that, if $\delta=(-A^{2}-A^{-2})$, then $\langle \widehat{Id_{t-2i}} \rangle^{p}=\delta^{p(t-2i-1)}$.
%, whereas $\langle \widehat{\Delta_{t-2i}^{2q}} \rangle$ is to complicated to get an easy formula.

\begin{proof}
The proof is analogous to the proof of the main theorem of \cite{C2}. By modifying the definition of the function $f$ and by verifying that it satisfies the same properties of the original one, we get our statement.
\end{proof}

\begin{ese}
Considering the two links $L_A$ and $L_B$ in $L(4,1)$ of Example~\ref{es2KBSM}, we can verify Formula~\ref{KBSMKB}. The lift of these two links is the Hopf link, that has Kauffman bracket equal to $-A^{4}-A^{-4}$.
Since $t=2$, the ideal $I$ is generated by $p$, $\delta^{3}-1$ and $ \{ \langle \widehat{Id_{2-2i}} \rangle^{4} - \langle \widehat{\Delta_{2-2i}^{2}} \rangle  \}_{i=0}$, that is to say, $p$, $-A^{6} + A^{2} - 1 + A^{-2} - A^{-6}$ and $\langle \widehat{Id_{2}} \rangle^{4} - \langle \widehat{\Delta_{2}^{2}} \rangle = \delta^{4} - (-A^{4}-A^{-4})= A^{8}+5A^{4}+6+5A^{-4}+A^{-8}$. A Groebner basis of $I$ is given by $\{4, A^{2} + 1 + A^{-2} \}$.

%\textcolor{red}{Now the mathematica program works correctly, we have to be careful to the writhe passing from the KBSM to the KB in S3 of La and Lb. And check if it is the same mod $1+a^2+a^4$.}
The closure in $\s3$ of the band diagram of $L_A$ is the trivial knot, with $\langle L_A \rangle = -A^{-3}$, while the same operation on the band diagram of $L_B$ produces $\langle L_B \rangle = (-A^{2}-A^{-2})$. 

%I used the notebook of mathematica named "Groebner" in our common folder.

For the knot $L_A$, Formula~\ref{KBSMKB} turns into
$$
(-A^{3})^{4} \equiv ( -A^{4}-A^{-4} ) \bmod I.
$$

For the link $L_B$, Formula~\ref{KBSMKB} turns into 

$$
(-A^2-A^{-2})^{4}= (A^{8}+4A^{4}+6+4A^{-4}+A^{-8})\equiv ( -A^{4}-A^{-4} ) \bmod I .
$$

\end{ese}

%========================================

\vspace{2mm}
\textit{Acknowledgments:} This research has been carried on at Universities of Bologna and Ljubljana. The authors are grateful to Matja Cencelj and Michele Mulazzani for promoting the collaboration.

%=========================================================

\vspace{15 pt} {BO\v{S}TJAN GABROV\v{S}EK, FME, University of Ljubljana, SLOVENIA. E-mail: bostjan.gabrovsek@fs.uni-lj.si}

\vspace{15 pt} {ENRICO MANFREDI, Department of Mathematics,
University of Bologna, ITALY. E-mail: enrico.manfredi3@unibo.it}

%%%%FINE DEL DOCUMENTO%%%%%%%%%%%%%%%%%%%%%%%%%%%%%%%%%%%%%%%%%%%%%%%%%%%%

\end{document}